\documentclass[pdflat,BibTeX,sn-mathphs]{sn-jnl}
\jyear{2021}
\theoremstyle{thmstyleone}
\newtheorem{theo}{Theorem}
\newtheorem{prop}{Proposition}
\usepackage{relsize,exscale}

\newtheorem{rem}{Remark}
\newtheorem*{rems}{Remarks}
\newtheorem{cor}{Corollary}
\newtheorem{lem}{Lemma}
\theoremstyle{thmstylethree}

\def\R{{\mathbb R}}
\def\N{{\mathbb N}}
\def\M{{\mathbb M}}
\begin{document}
	
	\title[\small Homogeneous incompressible Bingham viscoplastic as a limit of bi-viscosity fluids]{Homogeneous incompressible Bingham viscoplastic as a limit of bi-viscosity fluids}
	
	\author*[1,2]{\fnm{Wassim} \sur{Aboussi}}\email{aboussi@math.univ-paris13.fr}
	
	\author[1]{\fnm{Fayssal} \sur{Benkhaldoun}}\email{fayssal@math.univ-paris13.fr}
	
	\author[3]{\fnm{Ahmed} \sur{Aberqi}}\email{aberqi\_ahmed@yahoo.fr}
	
	\author[4]{\fnm{Abdallah} \sur{Bradji}}\email{abdallah.bradji@gmail.com}
	
	\author[2]{\fnm{Jaouad} \sur{Bennouna}}\email{jbennouna@hotmail.com}
	
	\affil*[1]{ Laboratory LAGA, CNRS, UMR 7539, Sorbonne Paris Nord University, Villetaneuse, France}
	
	\affil[2]{ Laboratory LAMA, Sidi Mohamed Ben Abdellah University, Faculty of Sciences Dhar El Mahraz, Fez, Morocco}
	
	\affil[3]{ Laboratory LAMA, Sidi Mohamed Ben Abdellah University, National School of Applied Sciences, Fez, Morocco}
	
	\affil[4]{ Laboratory LMA, University of Annaba, Faculty of Sciences, Annaba, Algeria}
	
	\abstract{ In this paper, the existence of a weak solution for homogeneous incompressible Bingham fluid is investigated. The rheology of such a fluid is defined by a yield stress $\tau_y$ and a discontinuous stress-strain law. {\color{black}This non-Newtonian fluid} behaves like a solid at low stresses and like a non-linear fluid above the yield stress.
	In this work we propose to build a weak solution for Navier stokes Bingham equations using a bi-viscosity fluid as {\color{black}an} approximation, in particular, we proved that the bi-viscosity tensor converges weakly to the Bingham tensor. This choice allowed us to show the existence of solutions for a given data $f\in L^2(0,T;V')$.}
	
	
	
	
	\maketitle
	\section{Introduction}
	\label{intro}
	As well known, the motion of a homogeneous incompressible fluid is governed by the Navier-Stokes system,
	which describes the balance of mass and momentum. The classical form of this equation is restricted to fluids
	whose stress-strain relationship is linear. This category of fluids is called Newtonian fluids. They have a
	simple molecular structure, e.g., water, air, and alcohol. The mathematical analysis of the Newtonian Navier
	Stokes equations {\color{black}are} one of the leading research topics that attract the attention of researchers because
	of the many open questions around this system (see \cite{Ref5, Ref9, Ref18,Ref19}).
 
 To study more complex fluids, such as molten plastics, synthetic {\color{black}fibres}, biological fluids, paints, and greases, etc., it is necessary to consider a generalized Navier Stokes system that models the behavior of fluids whose viscosity depends on the rate of deformation (i.e., non-Newtonian fluids). This complex behavior is translated into a mathematical complexity which gives rise to complex stress-strain laws, such as the Carreau-Yasuda, Bingham, power law, Cross, Casson, Herschel-Bulkley, etc., for more details on the rheology and the non-Newtonian models, consult \cite{Ref6, Ref12, Ref13}. \color{black} A rigorous \color{black} mathematical existence theory for non-Newtonian fluids can be found in \cite{Ref11}. 
Among the various classes of non-Newtonian materials, those exhibiting viscoplastic properties are particularly
interesting by their ability to strain only if the stress rate exceeds a minimum value. Many industrial processes
	involve viscoplastic fluid: mud, cement slurries, emulsions, foams, etc... The most commonly used model to account
	for this particular behavior is the Bingham model \cite{Ref1}. Eugene Bingham gave the initial mathematical
	expression in 1922 for one-dimensional flows.
	Later, Prager \cite{Ref15, Ref17} showed a generalized tensor formulation for multidimensional flows.
	From an analytical and numerical viewpoint, we cannot directly study the Navier Stokes Bingham problem since the stress tensor
	is unexplicit below the yield stress, moreover is a discontinuous operator (which prevents the use of \cite{Ref2})\color{black}.
	Duvaut and Lions \cite{Ref14} exclude the stress tensor by passing to a variational inequality for the velocity field to overcome these difficulties.
	Another solution was proposed by Basov and Shelukhin \cite{Ref8}, they proved the existence of weak solutions of the nonhomogeneous incompressible equation by using the Bercovier and Engelman model \cite{Ref20} as an approximation of the Bingham fluid.
	In \cite{Ref7}, Shelukhin used the same approach but with a different approximate tensor.
 
	Our work is based on the approximation of the Bingham tensor by the bi-viscosity tensor, which can be used for numerical simulation (see \cite{Ref3, Ref6, Ref16, Ref23}). Other regularization choices are possible, such as the Papanastasiou model \cite{Ref21} or the algebraic model proposed by Allouche et al. \cite{Ref22}.
	The reasons behind our choice is that the bi-viscosity operator is coercive, growing, monotonic and continuous,
	which are the conditions of an existence theorem given by \cite{Ref2}. The idea is to construct a sequence of approximate solutions using
	the bi-viscosity regularization and the theorem 1 \cite{Ref2}, then pass to the limit to prove the existence of a weak solution.
 
	In section \ref{sec:1}, we give the setting of the problem and the functional spaces, then we present our theorem and we give
	some remarks about the weak formulation. The proof is shared over three sections; the first step is provided in section \ref{sec:3}, where we
	propose an approximate problem and obtain a sequence of approximate solutions. The aim of section \ref{sec:4} is to prove various compactness
	results on the approximate solutions. Section \ref{sec:5} is devoted to passing to the limit in the approximate problem; in particular, \color{black} we prove that the bi-viscosity tensor converges weakly to the Bingham tensor\color{black}. In the last section, we prove the uniqueness of solutions.
	\section{Setting of the problem and main result}\label{sec:1}
	Let $\Omega$ be a smooth domain in $\R^2$ with Lipschitz boundary and $\Omega_T$ the open set $\Omega \times (0,T)$, where $T>0$ is the final time.
	\\ We consider an unsteady flow of incompressible Bingham fluid in 2D which is governed by the following Navier-Stokes system
	\begin{equation}\label{P}
	\left\{
	\begin{array}{cc}
	\displaystyle \partial_t u+(u\cdot\nabla)u-\nabla\cdot (\tau(Du))+\nabla p= f\hspace{20 pt}& \text{in }\Omega_T,\\
	\\
	\displaystyle \nabla\cdot  u=0 & \text{in }\Omega_T\cdot\\
	\end{array}
	\right.
	\end{equation}
	Here{\color{black},} $u$ is the velocity vector, $p$ is the pressure, and $\tau$ is the stress tensor where the strain tensor (shear tensor) is defined as $$ Du=\frac{1}{2}(\nabla u+\nabla u^t),$$ and $f:\Omega_T \rightarrow \R^2$
	represents the external forces (such as gravity). \color{black} The system \eqref{P} is equipped with the following initial condition \color{black}
	\begin{equation}\label{initial_data}
	u(\cdot,0)=u_0 \ \ \ \ \  \text{in } \Omega,
	\end{equation}
	\color{black}
	and the homogeneous Dirichlet boundary condition \color{black}
	\begin{equation}\label{boundary_data}
	u=0 \ \ \ \  \text{on } \partial\Omega\times(0,T)\cdot
	\end{equation}
	The Bingham stress–strain constitutive law is defined as
	\begin{equation}\label{Bingham_tensor}
	\left\{
	\begin{array}{ccc}
	\tau(Du)=\left(2\mu+\frac{\tau_y}{\mid Du \mid }\right)Du & & \text{if } \lvert \tau\rvert>\tau_y,\\
	\\
	Du=0   & &\text{if } {\color{black}\lvert \tau\rvert} \leq \tau_y\cdot
	\end{array}
	\right.
	\end{equation}
	Here{\color{black},} $\mu$ is the viscosity, $\tau_y$ is the yield stress and ${\color{black}\lvert A\rvert}^2=A:A${\color{black},} where the inner product is defined as
	$\displaystyle A:B=\sum_{i,j}^{}A_{ij}B_{ij}$. The Bingham tensor can be written as follows:
	\begin{equation}\label{Bingham_tensor2}
	\left\{
	\begin{array}{ccc}
	\tau(Du)=\left(2\mu+\frac{\tau_y}{{\color{black}\lvert Du \rvert} }\right)Du & & \text{if } Du \neq 0,\\
	\\
	{\color{black}\lvert \tau\rvert} \leq \tau_y   & &\text{if } Du=0\cdot
	\end{array}
	\right.
	\end{equation}
	Let us choose some spaces. Let $X$ be a Banach space, for each $1\leq p<\infty$, we defined the following function spaces :
	$$\displaystyle H=\left\{v \in L^2(\Omega) {\color{black},}\ \nabla\cdot v=0,\ \ \ \ v\cdot n \mid _{\partial \Omega }=0  \right\},$$ $$
	\displaystyle V=\left\{v \in H^1_0(\Omega) {\color{black},}\ \ \nabla\cdot v=0\right\}\cdot$$
	These two spaces are Hilbert spaces equipped with the scalar products respectively induced by
	those of $L^2(\Omega,\R^2)$ and of $H^1_0(\Omega,\R^2)$, i.e
	$$ {\color{black}\lVert v\rVert}_{H}^2=\int_{\Omega}^{} {\color{black}\lvert v\rvert} ^2 dx  \text{\hspace{20pt}and \hspace{17pt}}
	{\color{black}\lVert v\rVert}_{V}^2=\int_{\Omega}^{} {\color{black}\lvert \nabla v\rvert} ^2 dx,$$
	We also use the following Bochner spaces: {\color{black}$${L^p(0,T;X)=\left\{v \text{ measurable from (0,T) into $X$} ,\ {\color{black}\lVert v\rVert}_{L^p(0,T;X)}^p < \infty \right\},}$$
	$$ L^{\infty}(0,T;X)=\left\{v \text{ measurable from (0,T) into $X$} ,\ {\color{black}\lVert v\rVert}_{L^{\infty}(0,T;X)} < \infty \right\},$$
	where $\displaystyle {\color{black}\lVert v\rVert}_{L^p(0,T;X)}^p=\int_{0}^{T}{\color{black}\lVert v\rVert}_X^{p} $ and $\displaystyle {\color{black}\lVert v\rVert}_{L^{\infty}(0,T;X)}=\underset{t\in(0,T)}{\text{supess}}  {\color{black}\lVert v\rVert}_X$.}
	The space $E_{2,2}(V)=\left\{v \in L^2(0,T;V) ,\ \partial_t v \in L^2(0,T;V') \right\},$
	is a Banach space equipped with the norm $${\color{black}\lVert v\rVert}_{E_{2,2}} = {\color{black}\lVert v\rVert}_{L^2(0,T,V)}+ {\color{black}\lVert \partial_t v\rVert}_{L^2(0,T,V')} \cdot$$
	Where $V'$ is the topological dual of $V$, and we denote by $\langle \cdot , \cdot \rangle$ \color{black} the duality bracket
	between $V$ and $V'$.\color{black}\\
	
	\color{black}
	As in \cite{Ref2}, we call $(u,\tau(Du))\in E_{2,2} \times L^2(\Omega_T)$ a weak solution of the problem \eqref{P}-\eqref{Bingham_tensor}{\color{black},} if $u$ satisfies \eqref{initial_data} \color{black} and for all $\varphi\in L^2(0,T;V)$ we have 
	\begin{equation}\label{weak_formulation1}
	\int_{0}^{T}\langle \partial_tu,\varphi \rangle +\int_{\Omega_T}^{}\tau(Du):D\varphi\ +\int_{\Omega_T}^{}(u\cdot\nabla)u\cdot \varphi\ =
	\int_{0}^{T}\langle f,\varphi \rangle\ \cdot
	\end{equation}
	A similar formulation is given in \cite{Ref5}{\color{black},} for the Navier Stokes equation in 2D.\\
	
	\color{black}
	The main result of this work is the following theorem{\color{black}.}
	\begin{theo}\label{main_result}
		Assume that $f\in L^2(0,T;V')$ and $u_0 \in H$, then the Navier Stokes equation for a Bingham fluid \eqref{P}-\eqref{Bingham_tensor}{\color{black},}
		has \color{black} a weak solution \color{black} such that
		$$ u\in L^2(0,T;V)\cap L^{\infty}(0,T;H),\ \ \ \partial_tu\in L^2(0,T;V'),\ \ \ \tau(Du)\in L^2(\Omega_T).$$
	\end{theo}
	\vspace{-1cm}
	\begin{rems}
		\begin{enumerate}
			\item Theorem \ref{main_result}{\color{black},} ensure the existence of a classical weak solution $(u, p)\in E_{2,2}\times \mathcal{D'}(\Omega_T)$, for the system
			\eqref{P}-\eqref{Bingham_tensor}. Indeed, if we define the distribution {\color{black}
			$ T=\partial_t u+(u\cdot\nabla)u-\nabla\cdot (\tau(Du))-f$,}
			according to \eqref{weak_formulation1}, we can take $ \varphi \in \{\mathcal{D}(\Omega_T){\color{black},} \nabla\cdot \varphi=0\} $, and we have
			$\langle T,\varphi \rangle=0$.
			On the other hand, \color{black} the De Rham theorem\footnote{A constructive proof of the theorem is given by Simon in \cite{Ref10}.} \cite[p.~114]{Ref24} \color{black} ensures the existence of
			a primitive of any distribution that cancels on all test functions with null divergence (see \cite[th. IV.2.5]{Ref5}). Then, we obtain the
			existence of $p\in \mathcal{D'}(\Omega_T)$ where $T=-\nabla p$, which implies the existence of functions $(u, p)$ solution of \eqref{P}-\eqref{Bingham_tensor} in $\mathcal{D'}(\Omega_T)$.
			\item   We note that this weak formulation is different from the one proposed in \cite{Ref7}, where $f$ must belong to $L^2(\Omega_T)$, but in our case, $f$ belongs to $L^2(0,T,V')$.
			\item The Lions-Magenes theorem \cite{Ref5}{\color{black},} implies that the weak solution $u$ is continuous from $[0,T]$ into $H$.
		\end{enumerate}
	\end{rems}
	\section{Approximate solutions}\label{sec:3}
	In this section, we will build an approximate problem by regularizing the Bingham tensor \eqref{Bingham_tensor}{\color{black},} with another operator that approximates the physical behavior of Bingham fluids and has some analytical properties. The regularizing tensor is given by the bi-viscosity model :
	\begin{equation}\label{biv_tensor}
	\tau_m(A)=\left\{
	\begin{array}{cc}
	2m\mu A   &\text{if } {\color{black}\lvert A\rvert}  \leq \gamma_m,\\
	\\
	\left(2\mu+\frac{\tau_y}{{\color{black}\lvert A\rvert}  }\right)A & \text{if } {\color{black}\lvert A\rvert}  >\gamma_m\cdot
	\end{array}
	\right.
	\end{equation}
	Where $A\in\M^{2\times 2} $ and $\displaystyle \gamma_m=\frac{\tau_y}{2\mu(m-1)},\ m\geq 2$. The idea of this approximation is to consider
	the Bingham fluid when ${\color{black}\lvert \tau\rvert} \leq\tau_y$ (which is practically solid) as {\color{black}a} highly viscous Newtonian fluid, by involving a second artificial viscosity $\mu_m=m\mu$. Therefore, the equation \eqref{weak_formulation} can be viewed as an approximation of \eqref{weak_formulation1}.
 \vspace{-0.5 cm}
	\begin{theo}\label{approx_solutions}
		Assume that $f\in L^2(0,T;V')$ and $u_0 \in H$, then the approximate problem \eqref{P}-\eqref{boundary_data}, \eqref{biv_tensor}, has at least a solution $u_m\in E_{2,2}$ in the following sense :
		  \begin{equation}\label{weak_formulation}
		\int_{0}^{T}\langle \partial_tu_m,\varphi \rangle +\int_{\Omega_T}^{}\tau_m(Du_m):D\varphi +\int_{\Omega_T}^{}(u_m\cdot\nabla)u_m\cdot \varphi  =\int_{0}^{T}\langle f,\varphi \rangle,
		\end{equation}
		{\color{black} for all $\varphi\in L^2(0,T;V)$.} Moreover, $u_m$ is continuous from $[0,T]$ into $H$.
	\end{theo}
 \vspace{-0.5 cm}
 \begin{proof}
		This result is an application of theorem 1, proved by Dreyfuss and Hungerbühler in \cite{Ref2}, in other words, we will check the hypotheses (NS0)-(NS2) given in \cite{Ref2}.
       
		Clearly, $\tau_m$ satisfies (NS0) since it is a continuous function, which justifies the choice of $\gamma_m$. {\color{black} It is easy to prove that $\displaystyle \tau_m(A):A\geq 2\mu {\color{black}\lvert A\rvert} ^2$ and that $\displaystyle {\color{black}\lvert \tau_m(A) \rvert} \leq \tau_y+2\mu{\color{black}\lvert A\rvert} $, so $\tau_m$ satisfied the growth and the coercive hypotheses (NS1).}\\
		To prove the strict monotonicity of $\tau_m$, i.e.
		$\displaystyle \left( \tau_m(A)-\tau_m(B) \right):(A-B) > 0,$\\
          $ \forall A\neq B\in \mathbb{M}^{2\times 2},$
		we distinguish three cases: \\
		\textbf{Case 1: } if $\displaystyle {\color{black}\lvert A\rvert} \leq \gamma_m$, $\displaystyle {\color{black}\lvert B \rvert} \leq \gamma_m$ and $A \neq B$, then
		\begin{equation}\label{cas1}
		\left( \tau_m(A)-\tau_m(B) \right):(A-B) =2m\mu{\color{black}\lvert A-B \rvert} ^2>0\cdot
		\end{equation}
		\textbf{Case 2: } if $\displaystyle {\color{black}\lvert A\rvert} > \gamma_m$, $\displaystyle {\color{black}\lvert B \rvert} > \gamma_m$ and $A \neq B$ so
		\color{black}
		\begin{align*}
			\left( \tau_m(A)-\tau_m(B) \right):(A-B) &=\left( \left(\frac{\tau_y}{{\color{black}\lvert A\rvert} }+2\mu \right)A -\left(\frac{\tau_y}{{\color{black}\lvert B \rvert} }+2\mu \right)B \right):(A-B) \\
			&\hspace{-1cm}=2\mu {\color{black}\lvert A-B \rvert} ^2+\tau_y{\color{black}\lvert A\rvert} +\tau_y{\color{black}\lvert B \rvert} -\left(\frac{\tau_y}{{\color{black}\lvert A\rvert} }+\frac{\tau_y}{{\color{black}\lvert B \rvert} }\right)A:B\cdot
		\end{align*}
		By using the Cauchy–Schwarz inequality, we obtain
		$$ \left( \tau_m(A)-\tau_m(B) \right):(A-B) \geq 2\mu {\color{black}\lvert A-B \rvert} ^2+\tau_y{\color{black}\lvert A\rvert} +\tau_y{\color{black}\lvert B \rvert}$$ $$\hspace{3.5cm}-\left(\frac{\tau_y}{{\color{black}\lvert A\rvert} }+\frac{\tau_y}{{\color{black}\lvert B \rvert} }\right){\color{black}\lvert A\rvert} {\color{black}\lvert B \rvert} ,$$
		and we find
		\begin{equation}\label{cas2}
		\left( \tau_m(A)-\tau_m(B) \right):(A-B) \geq 2\mu {\color{black}\lvert A-B \rvert} ^2 >0\cdot
		\end{equation}
		\color{black}
		\textbf{Case 3: } if $\displaystyle {\color{black}\lvert A\rvert} > \gamma_m$ and $\displaystyle {\color{black}\lvert B \rvert} \leq \gamma_m$, so
		\begin{align*}
			(\tau_m(A)-\tau_m(B))&:(A-B)= \left(\left(2\mu+\frac{\tau_y}{{\color{black}\lvert A\rvert} } \right)A-2m\mu B \right):(A-B)& &\\
			& =\left(\left(2m\mu+\frac{\tau_y}{{\color{black}\lvert A\rvert} } \right)A-2m\mu B-2\mu(m-1)A\right):(A-B) & &\\
			& =2m\mu{\color{black}\lvert A-B \rvert} ^2+\left(\frac{\tau_y}{{\color{black}\lvert A\rvert} }-2\mu(m-1)\right)A:(A-B)\cdot& &
		\end{align*}
		On the other hand, we have $\displaystyle{\color{black}\lvert A\rvert} > \frac{\tau_y}{2\mu(m-1)}$, which gives, in addition to the Cauchy–Schwarz inequality :         
		{\color{black}\begin{equation}\label{cas3}
		\begin{split}
		\left( \tau_m(A)-\tau_m(B) \right)&:(A-B)\geq 2\mu{\color{black}\lvert A-B \rvert} ^2\\
		&+2(m-1)\mu{\color{black}\lvert A-B \rvert} \left({\color{black}\lvert A-B \rvert} +\frac{\tau_y}{2\mu(m-1)}-{\color{black}\lvert A\rvert} \right)\cdot 
		\end{split}
		\end{equation}}
		We also have $\displaystyle {\color{black}\lvert A-B \rvert} +\frac{\tau_y}{2\mu(m-1)}-{\color{black}\lvert A\rvert} \geq 0$, {\color{black} then
		$ \left( \tau_m(A)-\tau_m(B) \right):(A-B) >0\cdot $}
		Finally, we can apply {\color{black}Theorem} 1 of \cite{Ref2}, with $n=p=2$.
	\end{proof}
 \vspace{-0.5 cm}
	\begin{lem}\label{q}
		Form \eqref{cas1}, \eqref{cas2} and \eqref{cas3}{\color{black},} we deduce the following inequality
		\begin{equation}
		\left( \tau_m(A)-\tau_m(B) \right):(A-B) \geq 2\mu{\color{black}\lvert A-B \rvert} ^2, \hspace{30pt} \forall A,B \in \mathbb{M}^{2\times 2}\cdot 
		\end{equation}
		This inequality will be used somewhere in this paper.
	\end{lem}
	\section{Compactness of approximate solutions}\label{sec:4}
	The aim of this section is to prove some results on the sequence $u_m$.
 \vspace{-0.5 cm}
	\begin{prop}\label{estimations}
		The approximate solution $u_m$, constructed in Section~\ref{sec:3}{\color{black},} satisfied the following estimations \begin{itemize}
			\item [(i) ] The sequence $\displaystyle u_m$ is bounded in $\displaystyle L^2(0,T;V)\cap L^{\infty}(0,T;H)$.
			\item [(ii) ] The sequence $\displaystyle (u_m\cdot\nabla)u_m$ is bounded in $\displaystyle L^2(0,T;V')$.
			\item [(iii) ] The sequence $\displaystyle\tau_m (Du_m)$ is bounded in $\displaystyle L^2(\Omega_T)$.
			\item [(iv) ] The sequence $\displaystyle \partial_t u_m$ is bounded in $\displaystyle L^2(0,T;V')$.
		\end{itemize}
	\end{prop}
	In this paper, $c$ denotes various constants independent of $m$.
	\begin{proof}[Proof of (i)]
		By taking $u_m$ as a test function in the weak formulation \eqref{weak_formulation}, we obtain
		\begin{equation}\label{p_i}
		\underbrace{\int_{0}^{T}\langle \partial_tu_m,u_m \rangle}_{:=I_m^1}+\underbrace{\int_{\Omega_T}^{}\tau_m(Du_m):Du_m}_{:=I_m^2}+
		\underbrace{\int_{\Omega_T}^{}(u_m\cdot\nabla)u_m\cdot u_m}_{:=I_m^3}=\underbrace{\int_{0}^{T}\langle f,u_m \rangle}_{:=I_m}\cdot
		\end{equation}
		Let us start with the integral $I_m^1$, note that $u_m \in E_{2,2}$, so we use the Lions-Magenes theorem \cite{Ref5}
		$$ \displaystyle 2\int_{0}^{T}\langle \partial_tu_m,u_m \rangle={\color{black}\lVert u_m(T)\rVert}_H^2-{\color{black} \lVert u_0\rVert }_H^2,$$
		then,
		\begin{equation}\label{I_1}
		\displaystyle I_m^1 \geq -\frac{1}{2}{\color{black} \lVert u_0\rVert }_H^2\cdot
		\end{equation}
		Now{\color{black},} we will prove the existence of a constant $k >0 $ independent of m, such that
		\begin{equation}\label{I_2}
		I_m^2 \geq k {\color{black} \lVert u_m\rVert }_{L^2(0,T;V)}\cdot
		\end{equation}
		The coercivity of the operator $\tau_m$ implies
		\begin{equation}\label{11}
		\int_{\Omega}^{}\tau_m(Du_m):Du_m \geq 2\mu {\color{black} \lVert Du_m\rVert } ^2_{L^2(\Omega)}\cdot
		\end{equation}
		On the other hand, the Korn inequality\footnote{For more details, see chapter 2 of \cite{Ref11}.}
		ensures the existence of $K_{\Omega}>0$ such that
		\begin{equation}\label{22}
		{\color{black} \lVert \nabla u_m\rVert }^2_{L^2(\Omega)} \leq K_{\Omega}{\color{black} \lVert Du_m\rVert }^2_{L^2(\Omega)}\cdot
		\end{equation}
		By integrating the inequality \eqref{11} on $[0,T]$, and using \eqref{22} we find \eqref{I_2}.
		
		For the third integral, we have
		\begin{equation}\label{I_3}
		\int_{\Omega}^{} (u_m\cdot\nabla)u_m\cdot u_m=\frac{1}{2}\sum_{i}^{}\int_{\Omega}^{} u^i_m\frac{\partial}{\partial x_i}{\color{black}\lvert u_m \rvert} ^2dx =-\frac{1}{2}\int_{\Omega}^{}\nabla\cdot u_m {\color{black}\lvert u_m \rvert} ^2 dx=0\cdot
		\end{equation}
		We also have
		$$\displaystyle \int_{0}^{T}\langle f,u_m \rangle dt \leq \int_{0}^{T} {\color{black} \lVert f\rVert }_{V'}{\color{black} \lVert u_m\rVert }_V dt\cdot $$
		Using the $\varepsilon$-Young inequality {\color{black}with}  
$\varepsilon=k$ (the same $k$ in \eqref{I_2}) we obtain
		
		\begin{equation}\label{I}
		I_m\leq \frac{1}{2\varepsilon}\int_{0}^{T} {\color{black} \lVert f\rVert }_{V'}^2 +\frac{\varepsilon}{2}\int_{0}^{T}{\color{black} \lVert u_m\rVert }_V^2\cdot \end{equation}
		From \eqref{I_1}, \eqref{I_2}, \eqref{I_3} and \eqref{I} we deduce
		$$\varepsilon{\color{black} \lVert u_m\rVert }_{L^2(0,T;V)} \leq c+\frac{\varepsilon}{2}{\color{black} \lVert u_m\rVert }_{L^2(0,T;V)}+\frac{1}{2}{\color{black} \lVert u_0\rVert }_H^2\cdot$$
		We conclude that $u_m$ is bounded in $L^2(0,T;V)$.\\
		 
		Now we will show that $u_m$ is bounded in $L^{\infty}(0,T;H)$. Let $\theta \in (0,T]$, then the function given by $\varphi_m=u_m1_{[0, \theta]}$,
		can be a test function in the weak formulation \eqref{weak_formulation} and we obtain
		
		\begin{equation}
		\underbrace{\int_{0}^{T}\langle \partial_tu_m,\varphi_m \rangle}_{:=J_m^1}+\underbrace{\int_{\Omega_T}^{}\tau_m(Du_m):D\varphi_m}_{:=J_m^2}+
		\underbrace{\int_{\Omega_T}^{}(u_m\cdot\nabla)u_m\cdot \varphi_m}_{:=J_m^3}=\underbrace{\int_{0}^{T}\langle f,\varphi_m \rangle}_{:=J_m}\cdot
		\end{equation}
		As proved in the first part of this proof, we use the Lions-Magenes theorem
		\begin{equation}\label{J_1}
		J_m^1=\int_{0}^{\theta}\langle \partial_tu_m,u_m \rangle =\frac{1}{2} {\color{black} \lVert u_m(\theta)\rVert }_H^2-\frac{1}{2}{\color{black} \lVert u_0\rVert }_H^2,
		\end{equation}
		moreover, we have
		\begin{equation}\label{JJ}
		\displaystyle J_m^3=\int_{0}^{\theta}\int_{\Omega}^{}(u_m\cdot\nabla)u_m\cdot u_m =0,
		\end{equation}
		and thanks to the coercivity, we get
		\begin{equation}\label{swiri}
		J_m^2 =\int_{0}^{\theta}\int_{\Omega}^{}\tau_m(Du_m):Du_m \geq 0\cdot
		\end{equation}
		
		By using the Hölder inequality and the boundedness of $u_m$ in $L^2(0,T;V)$, we obtain
		\begin{equation}\label{J}
		J_m\leq {\color{black} \lVert f\rVert }_{L^2(0,T;V')}{\color{black} \lVert u_m\rVert }_{L^2(0,T;V)} \leq c\cdot
		\end{equation}
		From \eqref{J_1}, \eqref{JJ}, \eqref{swiri} and \eqref{J}, we deduce
		\begin{equation}\label{estimation1.0}
		\displaystyle{\color{black} \lVert u_m(\theta)\rVert }_H^2 \leq c+{\color{black} \lVert u_0\rVert }_H^2, \hspace{20pt} \forall \theta \in [0,T]\cdot
		\end{equation}
		Since $c$ is independent of $\theta$, the sequence $\displaystyle u_m$ is bounded in $L^{\infty}(0,T;H)$.\\
	\end{proof}
	
	\begin{proof}[Proof of (ii)]
		To prove this point, we need the following lemma:
		\begin{lem}\label{xx}
			The space $\displaystyle L^2(0,T;V)\cap L^{\infty}(0,T;H)$ is continuously embedded into $L^4(\Omega_T)$.
		\end{lem}
		Indeed, according to the lemma 6.2 \cite{Ref4} we have
		{\color{black}${\color{black}\lVert v\rVert}^2_{L^4(\Omega)}\leq c{\color{black}\lVert v\rVert}_{H^1_0}{\color{black}\lVert v\rVert}_{L^2}$, for any $v\in H^1_0(\Omega)\cdot$ Then we get}
		$${\color{black}\lVert v\rVert}^4_{L^4(\Omega_T)}\leq c {\color{black}\lVert v\rVert}_{L^{\infty}(0,T;H)}^2 {\color{black}\lVert v\rVert}_{L^2(0,T;V)}^2\cdot $$	
		So,  $\displaystyle L^2(0,T;V)\cap L^{\infty}(0,T;H)$ is continuously embedded into  $L^4(\Omega_T)$.\\
		Form lemma V.11 \cite{Ref5},
		\begin{equation}\label{vong}
		\int_{\Omega}^{}(u_m\cdot\nabla)u_m\cdot \varphi=-\int_{\Omega}^{}(u_m\cdot\nabla)\varphi\cdot u_m,  \hspace{20pt} \forall \varphi \in V\cdot
		\end{equation}
		Using Cauchy–Schwarz and Hölder inequalities, we obtain
		\begin{equation}\label{kk}
		\left\lvert \int_{\Omega}^{}(u_m\cdot\nabla)\varphi\cdot u_m  \right\lvert   \leq {\color{black} \lVert u_m\rVert }^2_{L^4}{\color{black} \lVert \nabla \varphi\rVert }_{L^2}\cdot
		\end{equation}
		Therefore, $$  {\color{black} \lVert (u_m\cdot\nabla)u_m\rVert }_{V'}\leq c {\color{black} \lVert u_m\rVert }^2_{L^4}\cdot$$
		Consequently,
		\begin{equation}\label{d}
		{\color{black} \lVert (u_m\cdot\nabla)u_m\rVert }^2_{L^2(0,T;V')}\leq c {\color{black} \lVert u_m\rVert }^4_{L^4(\Omega_T)}\cdot
		\end{equation}
		However, the sequence $u_m$ is bounded in $L^2(0,T;V)\cap L^{\infty}(0,T;H)$ and according to the Lemma \eqref{xx}, $u_m$ is bounded in $L^4(\Omega_T)$. Then $\displaystyle (u_m\cdot\nabla)u_m$ is bounded in $\displaystyle L^2(0,T;V')$.
	\end{proof}
	\begin{proof}[Proof of (iii)]
		Clearly, $(\tau_m(Du_m))_m$ is bounded in $L^2(\Omega_T)$. Indeed, we have
		$$\left\lvert \tau_m(Du_m)\right\rvert ^2\leq c(\tau_y^2+{\color{black}\lvert Du_m \rvert} ^2)\cdot $$
		Therefore{\color{black},} 
		$$ {\color{black} \lVert \tau_m(Du_m)\rVert }^2_{L^2}\leq c+c{\color{black} \lVert Du_m\rVert }^2_{L^2} \leq c+c{\color{black} \lVert u_m\rVert }^2_V \cdot$$
		by using the first estimation, we obtain
		\begin{equation}\label{yy}
		{\color{black} \lVert \tau_m(Du_m)\rVert }_{L^2(\Omega_T)} \leq c \cdot
		\end{equation}
	\end{proof}
	
	\begin{proof}[Proof of (iv)]
		Let us use again the weak formulation of the approximate problem. We have
		$$\left\lvert  \int_{0}^{T}\langle \partial_tu_m,\varphi \rangle\right\rvert   \leq \left\lvert  \int_{\Omega_T}^{}\tau_m(Du_m):D\varphi\right\rvert  +
		\left\lvert  \int_{\Omega_T}^{}(u_m\cdot\nabla)u_m\cdot
		\varphi\right\rvert  +\left\lvert  \int_{0}^{T}\langle f,\varphi \rangle\right\rvert\cdot $$
		
		By using the Hölder inequality for each integral, we obtain
		\begin{equation}\label{nnn}
		\left\{
		\begin{array}{cccc}
		\displaystyle \int_{0}^{T}\left\lvert \langle f,\varphi \rangle\right\rvert  \leq {\color{black} \lVert f\rVert }_{L^2(0,T;V')}{\color{black} \lVert \varphi\rVert }_{L^2(0,T;V)}, \\
		\displaystyle \int_{\Omega_T}^{}\left\lvert \tau_m(Du_m):D\varphi\right\rvert  \leq {\color{black} \lVert \tau_m(Du_m)\rVert }_{L^2(\Omega_T)}{\color{black} \lVert D\varphi\rVert }_{L^2(\Omega_T)}\cdot
		\end{array}
		\right.
		\end{equation}
		Thanks to $\eqref{kk}${\color{black},} and to the third estimation, we obtain
		\begin{equation}\label{ddm}
		\left\lvert  \int_{0}^{T}\langle \partial_tu_m,\varphi \rangle\right\rvert   \leq  c {\color{black} \lVert \varphi\rVert }_{L^2(0,T;V)}\cdot
		\end{equation}
		It follows that
		\begin{equation}\label{vv}
		{\color{black} \lVert \partial_tu_m\rVert }_{L^2(0,T;V')} \leq c\cdot
		\end{equation}
	\end{proof}
	\section{Passing to the limit}\label{sec:5}
	In this section, we will construct a weak solution of \eqref{P}-\eqref{Bingham_tensor} by using $\{u_m\}$ and some compactness results.
 \vspace{-0.5 cm}
	\begin{prop}\label{passage}
		The following convergence {\color{black}is} proved for subsequences which are denoted by $\displaystyle\{u_m\}$.
		\begin{itemize}
			\item [(i) ] $u_m \rightarrow u$\ \ \ weakly in $L^2(0,T;V)$ and {\color{black}weakly-*} in $L^{\infty}(0,T;H)$.
			\item [(ii) ] $\partial_t u_m \rightarrow \partial_t u$\ \ \ weakly in $L^2(0,T;V')$.
			\item [(iii) ] $(u_m\cdot\nabla)u_m \rightarrow (u\cdot\nabla)u$\ \ \ weakly in $L^2(0,T;V')$.
			\item [(iv) ] $\tau_m(Du_m) \rightarrow \tau(Du)$\ \ \ weakly in $L^2(\Omega_T)$.
		\end{itemize}
	\end{prop}
	Clearly, the function $u$ satisfy equation \eqref{weak_formulation1}. Moreover, It is easy to see that
	$\tau_m(Du_m)$ converges weakly to some $\xi$ in $L^2(\Omega_T)$ but the principal difficulty will be to show that $\xi$ is a Bingham tensor.
	\color{black}
	\begin{proof}[Proof of (i) ]
		The space $L^2(0,T;V)$ is reflexive, so from any bounded sequence, we can extract a subsequence which converges weakly in $L^2(0,T;V)$, then $u_m$ converges weakly to $u$ in $L^2(0,T;V)$.
		\color{black}
		On the other hand, the space $L^{1}(0,T;H)$ is separable \footnote{\color{black} For more details you can see \cite[Ch. 1]{Ref25} } which gives the {\color{black}weak-*} convergence in $L^{\infty}(0,T;H)$ of a subsequence of $u_m$, therefore we deduce (i).
		\color{black}
	\end{proof}
	\vspace{-0.5 cm}
	\begin{proof}[Proof of (ii).]
		We know that the differentiation operator with respect to time is continuous in the sense of distributions, it means $\partial_t{u_m} \longrightarrow \partial_t{u}${\color{black},} {\color{black}in the sense of distribution}. But we proved that $\partial_t{u_m}$ is bounded in $L^2(0,T;V')$ which implies the weak convergence in this space, therefore we deduce \color{black}(ii) \color{black} by the uniqueness of the limit in $\mathcal{D'}(\Omega_T)$.
	\end{proof}
	\vspace{-0.5 cm}
	\begin{proof}[Proof of (iii)]
		To prove this convergence we need the following strong convergence.
  \vspace{-0.5 cm}
		\begin{lem}\label{was1}
			The sequence $u_m$ converges strongly to $u$ in $L^2(0,T;H)$ and almost everywhere in $\Omega_T$.
		\end{lem}
  \vspace{-0.5 cm}
  
		This lemma is based on the compactness lemma (Theorem 5.1 \cite{Ref4}).\\
		We have $\partial_t u_m \rightarrow \partial_t u$ weakly in $L^2(0,T;V')$ and $ u_m \rightarrow u$ weakly in $L^2(0,T;V)$, using the compactness lemma, we obtain the strong convergence of $u_m$ to $u$ in $L^2(0,T;H)$. Moreover, we can extract a subsequence which converges to $u$ almost everywhere in $\Omega_T$.
		
		Now, we have to prove the weak convergence of $(u_m\cdot\nabla)u_m $ to $u\cdot\nabla u$ in $L^2(0,T;V')$. Due to the lemma \ref{was1}, {\color{black}$u_m \rightarrow u$ a.e in $\Omega_T$}, then for all $i,j\in \{1,2\}$ we have
		\begin{equation}\label{zzzz}
		u_m^iu_m^j \longrightarrow u^iu^j,\hspace{15pt}  \text{ a.e in } \Omega_T\cdot
		\end{equation}
		We also have
		$$\int_{\Omega}^{}\left(u_m^iu_m^j\right)^2dx \leq {\color{black} \lVert u_m^i \rVert }^2_{L^4} {\color{black} \lVert u_m^j \rVert } ^2_{L^4}\cdot  $$
		Since, $u_m$ is a bounded sequence in $L^4(\Omega_T)$, (Lemma\eqref{xx}), we obtain
		$${\color{black} \lVert u_m^iu_m^j\rVert }_{L^2(\Omega_T)}\leq c \cdot $$
		Which gives, by applying \color{black} Lemma 1.3 \cite[p. 12]{Ref4}\color{black}, the following convergence
		\begin{equation}\label{cv}
		u_m^iu_m^j\rightarrow u^iu^j \text{\hspace{10pt}weakly in }L^2(\Omega_T).
		\end{equation}
		Let $\varphi \in L^2(0,T;V) $, then
		$\displaystyle \int_{\Omega_T}^{}u_m^i\partial_iu_m^j\varphi_j=-\int_{\Omega_T}^{}u_m^iu_m^j\partial_i\varphi_j$ (according to \eqref{vong}). \eqref{cv}, permits to conclude that
		
		$$\int_{\Omega_T}^{}u_m^iu_m^j\partial_i\varphi_j \longrightarrow \int_{\Omega_T}^{}u^iu^j\partial_i\varphi_j,\hspace{20pt} \text{as } m\to \infty\cdot$$
		Consequently,
		$$\int_{\Omega_T}^{}u_m^i\partial_iu_m^j\varphi_j \longrightarrow \int_{\Omega_T}^{}u^i\partial_iu^j\varphi_j,\hspace{20pt} \text{as } m\to \infty\cdot$$
		Finally, we proved that
		\begin{equation*}\label{dd}
		\displaystyle\int_{\Omega_T}^{}(u_m\cdot\nabla)u_m\cdot \varphi = \sum_{i,j}^{2} \int_{\Omega_T}^{}u_m^i\partial_iu_m^j\varphi_j
		\longrightarrow \sum_{i,j}^{2} \int_{\Omega_T}^{}u^i\partial_iu^j\varphi_j = \int_{\Omega_T}^{}(u\cdot\nabla)u\cdot \varphi, \ \
		\end{equation*}
		for all $\varphi \in L^2(0,T;V)$.
		It follows that $(u_m\cdot\nabla)u_m$ converges to $(u\cdot\nabla)u$ weakly in $L^2(0,T;V')$.
	\end{proof}
	\begin{proof}[Proof of (iv)]
		To prove the weak convergence of $\tau_m(Du_m)$ to $\tau(Du)$, we start by proving that $Du_m$ converges strongly to $Du$ in $L^2(\Omega_T)$ (so almost everywhere in $\Omega_T$).
		\begin{lem}\label{minty}
			$$\int_{\Omega_T}^{} (\tau_m(Du_m)-\tau_m(Du)):(Du_m-Du)dxdt \longrightarrow 0, \text{\hspace{20pt}as } m\to +\infty\cdot  $$
		\end{lem}
		\begin{proof}
			Let us set the following notations :
			$$ I_m^1=\int_{\Omega_T}^{} \tau_m(Du_m):(Du_m-Du), \ \ I_m^2=\int_{\Omega_T}^{} \tau_m(Du):(Du_m-Du)\cdot $$
			We proved that $(u_m-u)\in L^2(0,T;V) $, so we can use $(u_m-u)$ as a test function in the weak formulation of the approximate problem, and we obtain
			\begin{equation*}
				\begin{split}
					\int_{0}^{T}\langle \partial_tu_m,u_m-u \rangle +\int_{\Omega_T}^{}\tau_m(Du_m):&D(u_m-u)+\int_{\Omega_T}^{}(u_m\cdot\nabla)u_m\cdot (u_m-u)\\
					&=\int_{0}^{T}\langle f,u_m-u \rangle,
				\end{split}
			\end{equation*}
			which implies that:
			$$I_m^1=\int_{\Omega_T}^{}\tau_m(Du_m):D(u_m-u)=J^1_m-J^2_m-J^3_m\cdot$$
			{\color{black}Where $$\displaystyle J^1_m =\int_{0}^{T}\langle f,u_m-u \rangle dt,
			\hspace{0.5cm}\displaystyle J^2_m =\int_{0}^{T}\int_{\Omega}^{}(u_m\cdot\nabla)u_m\cdot (u_m-u) dx dt,$$
			$$\text{and\ \ \ } \displaystyle J^3_m= \int_{0}^{T}\langle \partial_tu_m,u_m-u \rangle dt \cdot$$}
			{\color{black}Since}  $u_m \rightarrow u$ weakly in $L^2(0,T;V)${\color{black},} then $\displaystyle \lim_{m\rightarrow\infty}J^1_m =0 $.\\
			On other hand{\color{black},}   $\displaystyle J^2_m=-\int_{0}^{T}\int_{\Omega}^{}(u_m\cdot\nabla)u_m\cdot u $, and from convergence (iv),
			$$J^2_m \longrightarrow \int_{\Omega_T}^{}(u\cdot\nabla)u\cdot u =0\cdot$$
			For $J_m^3$, we use the Lions–Magenes theorem :
			$$ \frac{1}{2}{\color{black} \lVert u_m(T)-u(T)\rVert }^2_H =\int_{0}^{T}\langle \partial_t(u_m-u),u_m-u \rangle dt +
			\frac{1}{2}{\color{black} \lVert u_m(0)-u(0)\rVert }^2_H \cdot$$
			Moreover, $\displaystyle\int_{0}^{T}\langle \partial_tu,u_m-u \rangle dt \rightarrow 0$ as $m\rightarrow\infty$, this gives
			$$  \lim_{m\rightarrow\infty} \int_{0}^{T}\langle \partial_tu_m,u_m-u \rangle dt = \lim_{m\rightarrow\infty}
			\frac{1}{2}{\color{black} \lVert u_m(T)-u(T)\rVert }^2_H -
			\frac{1}{2}{\color{black} \lVert u_0-u(0)\rVert }^2_H\cdot $$
			To deduce that $u_0=u(0)$ in $H$, we will prove that $u_m(0) \longrightarrow u(0)$ weakly in $H$.\\
			We know that $E_{2,2}$ is continuously embedded into $C^0([0,T];H)$, then $u_m(0)$ is bounded in $H$. On the other hand, (i)  and (iii) of proposition \eqref{passage} imply that $u_m(0)$ converges weakly to $u(0)$ in $V'$. Consequently, we deduce that $u_m(0)\rightarrow u(0)$ weakly in $H$.
			Therefore, $\displaystyle\lim_{m\rightarrow\infty} J_m^3 \geq 0$, which implies that
			\begin{equation}
			\label{Im1} \lim_{m\rightarrow\infty} I_m^1 \leq 0\cdot
			\end{equation}
			
			Now, let us prove that $\displaystyle\lim_{m\rightarrow\infty} I_m^2 =0 $.
			We know that the sequence $u_m$ converges weakly to $u$ in  $L^2(\Omega_T)$, so, the sequence $Du_m$ converges to $Du$ in $D'(\Omega_T)$.
			In addition, $(Du_m)_m$ is bounded in $L^2(\Omega_T)$, then we deduce that  
			$$Du_m  \longrightarrow Du \text{\hspace{10pt}weakly in } L^2(\Omega_T)\cdot$$
			\color{black}
			On the other hand, $\tau_m(Du)$ converges strongly to $\phi$ in $L^2(\Omega_T)$, where:
			\begin{equation}\label{limsimple}
			\phi=\left\{
			\begin{array}{cc}
			\left(2\mu+\frac{\tau_y}{\mid Du \mid }\right)Du & \text{if } {\color{black}\lvert Du \rvert} >0, \\
			\\
			0   &\text{if }  Du=0\cdot
			\end{array}
			\right.
			\end{equation}
			Consequently
			\begin{equation}\label{ff}
			\lim_{m\rightarrow\infty}  \int_{\Omega_T}^{} (\tau_m(Du_m)-\tau_m(Du)):(Du_m-Du) \leq 0,
			\end{equation}
			which, with the strict monotonicity of $\tau_m$, gives
			$$\displaystyle \lim_{m\rightarrow\infty} \int_{\Omega_T}^{} (\tau_m(Du_m)-\tau_m(Du)):(Du_m-Du) =0\cdot $$
		\end{proof}
  \vspace{-0.5 cm}
		\begin{lem}
			$(Du_m)$ converges to $Du$ strongly in $L^2(\Omega_T)$ and a.e in $\Omega_T$.
		\end{lem}
  This lemma is a consequence of Lemma \ref{minty} and Lemma \ref{q}. Recall that we have the following inequality
		\begin{equation}\label{quasimono}
		2\mu {\color{black}\lvert Du_m-Du \rvert} ^2 \leq (\tau_m(Du_m)-\tau_m(Du)):(Du_m-Du)\cdot
		\end{equation}
		%
		Then, we deduce
		$$ \lim_{m\rightarrow \infty } {\color{black} \lVert Du_m-Du\rVert }^2_{L^2(\Omega_T)}\longrightarrow 0\cdot$$
		
		\color{black}
		We know that $\tau_m(Du_m)$ converges weakly to an element $\xi$ in $L^2(\Omega_T)$. So we must check that $\xi$ is a Bingham tensor.
		\color{black} The following proof is inspired by \cite{Ref7}, where Shelukhin et al. studies the Bingham problem
		with periodic boundary conditions.\color{black}\\
		
		We fix the following notations
		\begin{equation}\label{limbi}
		\Omega_T^+=\Omega_T\cap\{{\color{black}\lvert Du \rvert} >0\},\ \ \ \ \ \Omega_T^0=\Omega_T\cap\{{\color{black}\lvert Du \rvert} =0\}\cdot
		\end{equation}
		\textbf{Part 1: } Let us proof that ${\color{black}\lvert \xi \rvert} \leq \tau_y$ a.e in $\Omega_T^0$. Define
		$$ A=\Omega_T^0 \cap \{{\color{black}\lvert \xi \rvert} >\tau_y \}, \hspace{20pt} \varphi=\frac{\xi}{{\color{black}\lvert \xi \rvert} } 1_A,  \hspace{20pt} I=\int_{\Omega_T}^{}\xi:\varphi,$$
		$$I_m=\int_{\Omega_T}^{}\tau_m(Du_m):\varphi, \hspace{20pt}  a=I-\tau_y \text{{\color{black} meas}(A)}\cdot$$
		Suppose that $\text{{\color{black} meas}(A)} >0$, then $\displaystyle I=\int_{A}^{}{\color{black}\lvert \xi \rvert}  > \text{{\color{black} meas}(A)}\tau_y$, therefore $a>0$.\\
		On the other hand $I_m$ converges to $I$, i.e
		$$ \forall \varepsilon >0,\ \exists M(\varepsilon)\in\N :\ \ \forall m\geq M(\varepsilon),\ \ \ I-\varepsilon\leq I_m\leq I+\varepsilon \cdot $$
		We choose $\varepsilon=\frac{a}{2}$. Then, there exists $M(a)$, such that
		\begin{equation}\label{ddd}
		I_m\geq\frac{a}{2}+\tau_y \text{{\color{black} meas}(A)}, \hspace{20pt} \forall m\geq M(a)\cdot
		\end{equation}
		Let $m>\max(M(a),\eta)$, with $\displaystyle\eta=f_l\left(\frac{3\tau_y \text{{\color{black} meas}(A)}}{a}+1 \right)+1$, \color{black}where $f_l$ is the floor function. {\color{black}Furthermore}, we denote
		$$ A^1_m=\Omega_T\cap\{{\color{black}\lvert Du_m \rvert} \leq \gamma_m\},\ \ \ \ A^2_m=\Omega_T\cap\{\gamma_m < {\color{black}\lvert Du_m \rvert} \leq \gamma_{\eta}\}$$ $$\text{and\ \ \ } \ A^3_m=\Omega_T\cap\{{\color{black}\lvert Du_m \rvert} > \gamma_{\eta}\}.$$
		We have
		\begin{equation*}
			\begin{split}
				I_m=& \underbrace{\int_{A^1_m}^{}2m\mu Du_m:\varphi}_{:=I_m^1}+\underbrace{\int_{A^2_m}^{}\left(2\mu+\frac{\tau_y}{{\color{black}\lvert Du_m \rvert} }\right)
					Du_m:\varphi}_{:=I_m^2}\\ &+\underbrace{\int_{A^3_m}^{}\left(2\mu+\frac{\tau_y}{{\color{black}\lvert Du_m \rvert}} \right)Du_m:\varphi}_{:=I_m^3}\cdot
			\end{split}
		\end{equation*}
		Now calculate 
		\begin{equation}\label{basov}
		{\color{black}\lvert I_m^1 \rvert} \leq \int_{A^1_m\cap A}^{ }2m\mu {\color{black}\lvert Du_m \rvert}  \leq \frac{m}{m-1}\tau_y \text{{\color{black} meas}(}A^1_m\cap A\text{)}  ,
		\end{equation}
		\begin{equation}\label{basov2}
		\begin{split}
		{\color{black}\lvert I_m^2 \rvert}& \leq \tau_y \text{{\color{black} meas}(}A^2_m\cap A{)} +\int_{A^2_m\cap A}^{}2\mu {\color{black}\lvert Du_m \rvert}\\
		&\leq \frac{m}{m-1}\tau_y \text{{\color{black} meas}(}A^2_m\cap A\text{)} +2\mu \gamma_{\eta}\text{{\color{black} meas}(}\Omega_T{)} ,
		\end{split}
		\end{equation}
		\begin{equation}\label{basov3}
		{\color{black}\lvert I_m^3 \rvert} \leq \frac{m}{m-1}\tau_y \text{{\color{black} meas}(}A^3_m\cap A\text{)}+2\mu {\color{black} \lVert Du_m\rVert }_{L^2(\Omega_T)}\sqrt{\text{{\color{black} meas}(}A^3_m\cap A\text{)}} \cdot
		\end{equation}
		From (\ref{ddd}), (\ref{basov}), (\ref{basov2}) and (\ref{basov3}), we get
		\begin{equation}\label{wasssssssss}
		\begin{split}
		\frac{a}{2}+\tau_y \text{{\color{black} meas}(A)}&\leq\frac{m}{m-1}\tau_y \text{{\color{black} meas}(A)} +2\mu\gamma_{\eta}\text{{\color{black} meas}(}\Omega_T\text{)} \\
		&	+2\mu{\color{black} \lVert Du_m\rVert } _{L^2(\Omega_T)}\sqrt{\text{{\color{black} meas}(}A^3_m\cap A\text{)}} )\cdot
		\end{split}
        \end{equation}
		Due to the choice of $\eta$, we obtain $\displaystyle\gamma_{\eta}< \frac{a}{6\mu \text{{\color{black} meas}(}\Omega_T\text{)}} $, and we have
		$\text{{\color{black} meas}(}A\cap A^3_m\text{)} \rightarrow 0 $, so
		\begin{equation}\label{conta}
		\frac{a}{2}+\tau_y  \text{{\color{black} meas}(A)}\leq \tau_y  \text{{\color{black} meas}(A)} +\frac{a}{3}\cdot
		\end{equation}
		\color{black} Which is absurd, i.e $ \text{{\color{black} meas}(A)}=0 $, thus  ${\color{black}\lvert \xi \rvert} \leq \tau_y $ a.e in $\Omega_T^0$.\\
		\color{black}
		\textbf{Part 2: } Let us proof that $\xi=\tau(Du)$ a.e in $\Omega_T^+$.\\
		Set
		$$ B^1_m=\Omega_T^+ \cap \{{\color{black}\lvert Du_m \rvert} \leq\gamma_m\}\hspace{10pt} \text{and}\hspace{10pt}  B^2_m=\Omega_T^+ \cap \{{\color{black}\lvert Du_m \rvert} >\gamma_m\}.$$
		We have $$W_m:= \left\lvert \tau_m(Du_m)-\tau(Du)\right\rvert  1_{\Omega_T^+}=\left\lvert 2\mu Du_m 1_{B^1_m}+F(Du_m) 1_{B^2_m}-F(Du) 1_{\Omega_T^+} \right\rvert , $$
		where, $\displaystyle F(A)=\left(2\mu+\frac{\tau_y}{{\color{black}\lvert A\rvert} }\right)A$,
		then $$
		W_m\leq \frac{m}{m-1}\tau_y  1_{B^1_m}+\left\lvert F(Du_m) 1_{B^2_m}-F(Du) 1_{\Omega_T^+}\right\rvert  \cdot$$
		However, $Du_m \rightarrow Du$ a.e in $\Omega_T^+$ and the function $X\mapsto F(X)$ is continuous, then $F(Du_m) \rightarrow F(Du)$ a.e in $\Omega_T^+ $. On the other hand \color{black} $ 1_{B^1_m}\rightarrow 0 $ and $ 1_{B^2_m}\rightarrow  1_{\Omega_T^+} ${\color{black},} \color{black} which gives $W_m \rightarrow 0$, i.e.
		$$ \tau_m(Du_m) \rightarrow \tau(Du) \text{\ \ \ \  a.e in } \Omega_T^+\cdot$$
        Let $\psi \in L^{\infty}(\Omega_T)$ be such that $\psi_{\mid _{\Omega_T^0}}=0 $. Let $Q'\subset \Omega_T$, $\theta_m=\tau_m(Du_m):\psi${\color{black},} and $\theta=\tau(Du):\psi$.\\
    	Using the Hölder inequality, we obtain
		\begin{equation}
		\int_{Q'}^{}{\color{black}\lvert \theta_m  \rvert}
		\leq {\color{black}\lvert {\color{black}\lvert \psi \rvert}  \rvert} _{L^{\infty}(\Omega_T)}\sqrt{\text{{\color{black} meas}(}Q')}\left(\tau_y \sqrt{\text{{\color{black} meas}(}\Omega_T)}+2\mu{\color{black} \lVert Du_m\rVert } _{L^2(\Omega_T)}\right)\cdot
		\end{equation}
		Therefore $\theta_m$ is uniformly integrable on $\Omega_T$ and $\theta_m \rightarrow \theta $ a.e in $\Omega_T$.
		This gives, thanks to Vitali theorem,
		$\displaystyle \int_{\Omega_T^+}^{}\tau_m(Du_m):\psi \rightarrow \int_{\Omega_T^+}^{}\tau(Du):\psi $.\\
		On the other hand $\tau_m(Du_m)$ converges weakly to $\xi$ in $L^2(\Omega_T)$, then $\tau(Du)=\xi $ a.e in $\Omega_T^+$.\\
		Finally, we proved that $\tau_m(Du_m)$, converges weakly to a Bingham tensor and the proof is completed.
	\end{proof}
	\color{black}
	\section{Uniqueness of solutions}\label{sec:6}
	In this section{\color{black},} we will prove that the problem \eqref{P}-\eqref{Bingham_tensor} has a unique solution. To do this we are inspired by the uniqueness proof of the Newtonian Navier Stokes equation.\\
	We consider $u_1$ and $u_2$ to be two weak solutions of \eqref{weak_formulation1} and introduce $u=u_1-u_2$. Therefore, we obtain
	\begin{equation}\label{unicite1}
    \begin{split}
    \int_{0}^{T}\langle \partial_tu,\varphi \rangle +\int_{\Omega_T}^{}(\tau(Du_1)-\tau(Du_2))&:D\varphi\ +\int_{\Omega_T}^{}(u_1\cdot\nabla)u_1\cdot \varphi \\
    -\int_{\Omega_T}^{}(u_2\cdot\nabla)u_2&\cdot \varphi =0, \hspace{1cm}
    \forall \varphi \in L^2(0,T;V)\cdot
    \end{split}  
	\end{equation}
	 On the other hand, we have$$
	\int_{\Omega_T}^{}(u_1\cdot\nabla)u_1\cdot \varphi\ -\int_{\Omega_T}^{}(u_2\cdot\nabla)u_2\cdot \varphi = \int_{\Omega_T}^{}(u_2\cdot\nabla)u\cdot \varphi\ +\int_{\Omega_T}^{}(u\cdot\nabla)u_1\cdot \varphi \cdot
	$$
	Let $t\in (0,T)$. {\color{black} Taking the function $\displaystyle\varphi=u1_{[0,t]}$ in \eqref{unicite1}  yields}
	\begin{equation}\label{unicite2}
	\begin{split}
&\int_{0}^{t}\langle \partial_tu,u\rangle ds+\int_{0}^{t}\int_{\Omega}^{}(\tau(Du_1)-\tau(Du_2)):Du\ dx ds  
	\\
	&+\int_{0}^{t}\int_{\Omega}^{}(u\cdot\nabla)u_1\cdot u\ dx ds+\int_{0}^{t}\underbrace{\int_{\Omega}^{}(u_2\cdot\nabla)u\cdot u\ dx}_{=0} ds=0\cdot
	\end{split}
	\end{equation}
	Using {\color{black}the Lions-Magenes Theorem} we obtain
	\begin{equation}\label{zz}
    \begin{split}
     \frac{1}{2}{\color{black} \lVert u(s)\rVert }_{H}^2+\int_{0}^{t}\int_{\Omega}^{}(u\cdot\nabla)u_1\cdot u +&\int_{0}^{t}\int_{\Omega}^{}(\tau(Du_1)-\tau(Du_2)):Du \\
     &=\frac{1}{2}{\color{black} \lVert u(0)\rVert }_{H}^2\cdot
    \end{split}
	\end{equation}
	According to \eqref{kk} and Lemma\eqref{xx} we have
	\begin{equation}\label{srr}
	\int_{\Omega}^{}\left\lvert(u\cdot\nabla)u_1\cdot u\ \right\rvert
	dx \leq c {\color{black}\lVert u\rVert}_V {\color{black}\lVert u\rVert}_H {\color{black} \lVert u_1\rVert }_V\cdot
	\end{equation}
	Furthermore, we can easily prove the following inequality\footnote{ We can adapt the proof of the strict monotonicity of $\tau_m$.}
	\begin{equation}\label{taumonono}
	\left( \tau(A)-\tau(B) \right):(A-B) \geq 2\mu{\color{black}\lvert A-B \rvert}^2, \hspace{30pt} \forall A,B \in \mathbb{M}^{2\times 2}\cdot
	\end{equation}
	From \eqref{taumonono} and Korn's inequality we obtain
	\begin{equation}\label{badja}
	\int_{\Omega}^{}(\tau(Du_1)-\tau(Du_2)):Du\ dx \geq \frac{2\mu}{K_{\Omega}} {\color{black}\lVert u\rVert}^2_V \cdot
	\end{equation}
 	Thus,{\color{black}\begin{equation}\label{zb}
    \begin{split}
     \frac{1}{2}{\color{black} \lVert u(s)\rVert }_{H}^2+\frac{2\mu}{K_{\Omega}}\int_{0}^{t} {\color{black}\lVert u\rVert}^2_V  ds \leq &c\int_{0}^{t} {\color{black}\lVert u\rVert}_V {\color{black}\lVert u\rVert}_H {\color{black} \lVert u_1\rVert }_V \ ds \\
     & +\frac{1}{2}{\color{black} \lVert u(0)\rVert }_{H}^2\cdot
    \end{split}
	\end{equation}}
	Using Young’s inequality, we get
	\begin{equation}\label{zbb}
	{\color{black} \lVert u(s)\rVert }_{H}^2 \leq {\color{black} \lVert u(0)\rVert }_{H}^2+c\int_{0}^{t} {\color{black}\lVert u\rVert}_H^2 {\color{black} \lVert u_1\rVert }_V^2 \ ds\cdot
	\end{equation}
	Thanks to the Gronwall lemma, we deduce that
	$${\color{black} \lVert u(s)\rVert }_{H}^2 \leq {\color{black} \lVert u(0)\rVert }_{H}^2 \exp \left(c \int_{0}^{t} {\color{black} \lVert u_1\rVert }_V^2 \ ds\right),\ \ \ \ \ \forall t\in [0,T]\cdot$$
	Since $u(0)=0$, we get the uniqueness of the weak solutions.
	\begin{cor}[Energy equality]
		The solution $u$ is more {\color{black}than} a classical weak solution. In fact, we have $u\in C^0([0,T];H)$, moreover, {\color{black}for all $s_1, s_2 \in [0,T]$,} $u$ satisfies the following energy equality
		\begin{equation}\label{energy}
		\frac{1}{2}{\color{black} \lVert u(s_2)\rVert }_{L^2(\Omega)}^2+\int_{s_1}^{s_2}\int_{\Omega}^{}\tau(Du):Du=\int_{s_1}^{s_2}
		\langle f,u \rangle+ \frac{1}{2}{\color{black} \lVert u(s_1)\rVert }_{L^2(\Omega)}^2\cdot
		\end{equation}
	\end{cor}
	To prove the energy equality, we have only to take $\displaystyle\varphi=u1_{[s_1,s_2]}$ as a test function in  \eqref{weak_formulation1} and use the Lions-Magenes theorem.
	\begin{cor}[Variational inequality]
		The weak solution given by Theorem \ref{main_result} satisfies the following variational inequality,  for all $\varphi\in L^2(0,T;V)$
		\begin{equation*}
			\int_{0}^{T}\langle \partial_tu,\varphi-u \rangle +\int_{\Omega_T}^{}(u\cdot\nabla)u\cdot \varphi+2\mu\int_{\Omega_T}^{}Du:D(\varphi-u) +\tau_y\int_{\Omega_T}^{}({\color{black}\lvert D\varphi \rvert} -{\color{black}\lvert Du \rvert} )
		\end{equation*}
		\begin{equation}\label{inequality}
		\geq \int_{0}^{T}\langle f,\varphi-u \rangle\cdot
		\end{equation}
	\end{cor}
	\begin{proof}
		Let us show the following inequality
		\begin{equation}\label{didi}
		\int_{\Omega_T}^{}\tau(Du):D(\varphi-u)\leq 2\mu \int_{\Omega_T}^{}Du:D(\varphi-u)+\tau_y\int_{\Omega_T}^{}{\color{black}\lvert D\varphi \rvert} -\tau_y\int_{\Omega_T}^{}{\color{black}\lvert Du \rvert} \cdot\end{equation}
		Using the Cauchy–Schwarz inequality and with the notation \eqref{limbi}, we obtain
		\begin{eqnarray*}\
			\int_{\Omega_T^+}^{}\tau(Du):D(\varphi-u)& = &\int_{\Omega_T^+}^{}\left(2\mu+\frac{\tau_y}{{\color{black}\lvert Du \rvert} }\right)Du:D(\varphi-u)\\
			& =&2\mu \int_{\Omega_T}^{}Du:D(\varphi-u)+\tau_y \int_{\Omega_T^+}^{}\frac{Du:D\varphi}{{\color{black}\lvert Du \rvert} }- \tau_y \int_{\Omega_T}^{}{\color{black}\lvert Du \rvert} \hspace{40pt} \\
			& \leq &2\mu \int_{\Omega_T}^{}Du:D(\varphi-u)+\tau_y \int_{\Omega_T^+}^{}{\color{black}\lvert D\varphi \rvert} - \tau_y \int_{\Omega_T}^{}{\color{black}\lvert Du \rvert} \cdot
		\end{eqnarray*}
		We also have
		\begin{equation}\label{dddd}
		\int_{\Omega_T^0}^{}\tau(Du):D(\varphi-u)=\int_{\Omega_T^0}^{}\tau(Du):D\varphi \leq \tau_y \int_{\Omega_T^0}^{}{\color{black}\lvert D\varphi \rvert} \cdot
		\end{equation}
		Hence, we deduce the inequality \eqref{didi}. {\color{black}This} implies, jointly with \eqref{weak_formulation1}, the variational inequality.
	\end{proof}
	\begin{rem}
		\color{black}
		The inequality \eqref{inequality} implies that $u$ satisfies the variational inequality proposed by Lions and Duvaut in \cite{Ref14}, i.e.
       {\color{black}
		\begin{equation}\label{inequality2}
		\begin{split}
		\langle \partial_t u(t),\varphi-u(t) \rangle +&\int_{\Omega}^{}(u(t)\cdot\nabla)u(t)\cdot \varphi+2\mu\int_{\Omega}^{}Du(t):D(\varphi-u(t))\\
		&+\tau_y\int_{\Omega}^{}({\color{black}\lvert D\varphi \rvert} -{\color{black}\lvert Du(t) \rvert} ) \geq \langle f,\varphi-u(t) \rangle,
		\end{split}
		\end{equation}}
		for any $\varphi \in V$. The proof of this result is given in \cite[p. 300-301]{Ref14}.
		\color{black}
	\end{rem}
	\section{Conclusion and outlook}
	\label{conclusion and Remarks}
As mentioned in the introduction, this work aims to prove the existence of
the Navier Stokes equation solution for an incompressible homogeneous fluid that follows the Bingham model. In {\color{black}the} first step, we constructed an approximate problem using the bi-viscosity model, which behaves like a Newtonian fluid under weak stress and like a non-Newtonian fluid when the stress rate is great than the yield stress. After this approximation, we applied the theorem presented by Dreyfuss and Hungerbuhler in \cite{Ref2}, {\color{black}and} then a weak solution {\color{black}to} the problem in question was constructed by passing to the limit. This analysis shows that the conditions of Theorem 1 \cite{Ref2} {\color{black}is} sufficient but not necessary since the Bingham tensor does not satisfy them. Another essential advantage of our theorem is that the membership of the function $f$ to the space $L^2(\Omega_T)$ is not necessary (which is the case in \cite{Ref7, Ref8}).
The next objective is to extend Theorem \ref{main_result} to a thixotropic Bingham model, i.e., the yield strength is linearly dependent on the structural parameter, which follows a first-order rate equation taking into account the decay and accumulation of the material structure.
The study of the non-homogeneous case may also be the subject of future work. The convergence of the Bingham solution to The Newtonian solution, when $\tau_y\to 0$, can be proved. A long-term objective is to analyze the non-Newtonian Navier Stokes equation, more complicated than the Bingham model, as Herschel–Bulkley and Casson models.
    \bibliographystyle{abbrv}
	\bibliography{sn-bibliography}
\end{document}